\documentclass{svproc}
\usepackage{url}

\usepackage{amsfonts}
\usepackage{amssymb}
\usepackage{amsmath}
\usepackage{amssymb}
\usepackage{mathrsfs}
\usepackage{graphicx}
\usepackage{caption} 
\usepackage{wrapfig}
\usepackage{xcolor}
\usepackage{dsfont}
\usepackage{empheq}
\usepackage{thmtools}
\usepackage{thm-restate}
\usepackage[normalem]{ulem}

\usepackage{hyperref}
\hypersetup{
  unicode = true,                           
  pdftoolbar = true,                        
  pdfmenubar = true,                        
  pdffitwindow = true,                      
  pdfauthor = {J. Moatti},              
  pdftitle = {A skeletal high-order structure-preserving scheme for advection-diffusion equations},         
  pdfnewwindow = true,                      
  colorlinks = true,                        
  linkcolor = blue,                         
  citecolor = magenta,                      
  filecolor = violet,                       
  urlcolor = violet                         
}

\usepackage{authblk}

\usepackage{amsmath,amsfonts,amssymb,mathtools}
\usepackage{upgreek}
\usepackage{nicefrac}
\usepackage{cancel}
\usepackage{comment}
\usepackage{enumerate}
\usepackage{graphicx}
\usepackage{tikzscale}
\usepackage{xcolor}
\usepackage{paralist}
\usepackage{pgfplots}
\pgfplotsset{every axis/.append style={
                    label style={font=\Large},
                    tick label style={font=\Large},
                    legend style={font=\Large}
                    }}
\usepackage{subcaption}
\usepackage{etoolbox}
\usepackage[skins]{tcolorbox}
\makeatletter
\patchcmd{\ttlh@hang}{\parindent\z@}{\parindent\z@\leavevmode}{}{}
\patchcmd{\ttlh@hang}{\noindent}{}{}{}
\makeatother
\usepackage{multirow}
\usetikzlibrary{patterns}

\usepackage{tabularx}
\newcolumntype{Y}{>{\centering\arraybackslash}X}

\pgfplotsset{compat=1.15}
\newcommand{\logLogSlopeTriangle}[5]
{
    \pgfplotsextra
    {
        \pgfkeysgetvalue{/pgfplots/xmin}{\xmin}
        \pgfkeysgetvalue{/pgfplots/xmax}{\xmax}
        \pgfkeysgetvalue{/pgfplots/ymin}{\ymin}
        \pgfkeysgetvalue{/pgfplots/ymax}{\ymax}

        \pgfmathsetmacro{\xArel}{#1}
        \pgfmathsetmacro{\yArel}{#3}
        \pgfmathsetmacro{\xBrel}{#1-#2}
        \pgfmathsetmacro{\yBrel}{\yArel}
        \pgfmathsetmacro{\xCrel}{\xArel}

        \pgfmathsetmacro{\lnxB}{\xmin*(1-(#1-#2))+\xmax*(#1-#2)} 
        \pgfmathsetmacro{\lnxA}{\xmin*(1-#1)+\xmax*#1} 
        \pgfmathsetmacro{\lnyA}{\ymin*(1-#3)+\ymax*#3} 
        \pgfmathsetmacro{\lnyC}{\lnyA+#4*(\lnxA-\lnxB)}
        \pgfmathsetmacro{\yCrel}{\lnyC-\ymin)/(\ymax-\ymin)}

        \coordinate (A) at (rel axis cs:\xArel,\yArel);
        \coordinate (B) at (rel axis cs:\xBrel,\yBrel);
        \coordinate (C) at (rel axis cs:\xCrel,\yCrel);

        \draw[#5]   (A)-- node[pos=0.5,anchor=north] {\scriptsize{1}}
                    (B)-- 
                    (C)-- node[pos=0.,anchor=west] {\scriptsize{#4}} 
                    (A);
    }
}

\newtheorem{prop}{Proposition}

\DeclareMathOperator{\e}{e}
\DeclareMathOperator{\divergence}{div}

\newcommand{\R}{\mathbb{R}}

\newcommand{\M}{\mathcal{M}}
\newcommand{\E}{\mathcal{E}}
\newcommand{\D}{\mathcal{D}}
\newcommand{\Diss}{\mathbb{D}}
\newcommand{\Entro}{\mathbb{E}}

\newcommand{\poly}{\mathbb{P}}

\newcommand{\G}{\mathcal{G}}

\newcommand{\s}{\sigma}

\renewcommand{\u}{\mbox{\uwave{$u$}}}

\newcommand{\V}{\underline{V}}
\renewcommand{\v}{\underline{v}}
\newcommand{\w}{\underline{w}}

\newcommand{\phid}{\underline{\phi}}

\renewcommand{\l}{\underline{\ell}}

\newcommand{\1}{\mathds{1}}
\newcommand{\one}{\underline{1}}

\definecolor{MyGreen}{RGB}{54,165,54}

\begin{document}
\mainmatter              
\title{A skeletal high-order structure preserving scheme for advection-diffusion equations}
\titlerunning{An high-order structure preserving  scheme for advection-diffusion}  
%
\author{Julien Moatti}
\authorrunning{Julien Moatti} 
%
\tocauthor{Julien Moatti}
\institute{ Inria, Univ. Lille, CNRS, UMR 8524 - Laboratoire Paul Painlev\'e, F-59000 Lille, France \\
\email{julien.moatti@inria.fr}
}

\maketitle              

\begin{abstract}
We introduce a nonlinear structure preserving high-order scheme for anisotropic advection-diffusion equations.
This scheme, based on Hybrid High-Order methods, can handle general meshes. 
It also has an entropy structure, and preserves the positivity of the solution. 
We present some numerical simulations showing that the scheme converges at the expected order, while preserving positivity and long-time behaviour.  
\end{abstract}
\section{Motivations and context}
We are interested in the discretisation of a linear advection-diffusion equation on general meshes with a high-order scheme.
Let $\Omega$ be an open, bounded, connected polytopal subset of $\R^d$, $d \in \{2,3\}$.
We consider the following problem with homogeneous Neumann boundary conditions: find $u : \R_+ \times \Omega \to \R$ solution to
\begin{equation} \label{pb:evol}
	\left\{
	\begin{split}
		\partial _ t u - \divergence ( \Lambda  (\nabla u + u \nabla \phi )  ) &= 0 &&\text{ in } \R_+ \times \Omega, \\
		\Lambda (\nabla u + u \nabla \phi ) \cdot n &= {0} &&\text{ on } \R_+ \times \partial \Omega,\\		
	 	u(0,\cdot)&= u^{in} &&\text{ in } \Omega ,
	\end{split}
	\right.
\end{equation}
where $n$ is the unit normal vector to $\partial\Omega$ pointing outwards {from} $\Omega$. We assume that the data satisfy:
(i) $\Lambda \in L^\infty(\Omega;\R^{d\times d})$ is a uniformly elliptic diffusion tensor: 
there exists $\lambda_\flat > 0$ such that,
for a.e.~$x$ in $\Omega$, $ \Lambda (x) \xi \cdot \xi \geq \lambda_\flat|\xi|^2$ for all $\xi \in \R^d$; 
(ii) $\phi \in C^1(\overline{\Omega})$ is a regular potential;
(iii) $u^{in} \in L^1(\Omega)$ is a non-negative initial datum, such that $\int_\Omega u^{in} \log\left(u^{in}\right) < \infty$.
The solutions to \eqref{pb:evol} enjoy some specific and well-known properties. 
First the mass is preserved along time, i.e. for almost every $t>0$, $\int_{\Omega}u(t)= \int_\Omega u^{in} = M$ where $M >0$ is the initial mass. 
Second, the solution is positive {for $t>0$}. 
Last, the solution has a specific long-time behaviour:  
it converges exponentially fast when $t \to \infty$ towards the thermal equilibrium $u^\infty$, solution to the stationary problem associated to \eqref{pb:evol}, defined as $u^\infty= \frac{M }{\int_\Omega \e^{-\phi}}\e^{-\phi} $.

%
%
In order to get a reliable numerical approximation of such problems, one has to preserve these structural properties at the discrete level.
It is {well-known that two-point} finite volume methods are structure preserving (see \cite{CHH:20} for the long-time behaviour),
but these methods can only be used for isotropic tensors on meshes satisfying some orthogonality conditions.
On the other hand, finite volume methods (using auxiliary unknowns) for anisotropic problems on general meshes
were introduced in the past twenty years, but none of these linear methods preserve the positivity of the solutions (see \cite{Droni:14}).
A possible alternative was proposed in \cite{CaGui:17}, {with the introduction and analysis of a nonlinear positivity preserving Vertex Approximate Gradient VAG scheme.}  
Following these ideas, a nonlinear Hybrid Finite Volume (HFV) scheme was designed in \cite{CHHLM:22}.
\newline
All the schemes discussed above are at most of order two in space {(in $L^2$ norm)}. 
The aim of this paper is to introduce a high-order scheme preserving the three structural properties discussed above.
Since the HFV method coincides with the low-order version of the Hybrid High-Order (HHO) scheme introduced in \cite{DiPEL:14}, we propose an HHO generalisation of the scheme introduced in \cite{CHHLM:22}.
Numerical results indicate that this scheme offers a better efficiency in terms of computational cost than low order schemes. 
%
%
%
%
%

\section{Discrete setting and scheme}
\subsection{Mesh}
We define a discretisation of $\Omega$ as a {pair} $\mathcal{D} = ( \M, \E)$, where:
\begin{itemize}
		\item the mesh $\M$ is a partition of $\Omega$ into \emph{cells}, i.e., a finite family of nonempty disjoint open polytopal subsets $K$ of $\Omega$ such that $\overline{\Omega} = \bigcup _ {K \in \M } \overline{K}$, 
%
%
		\item the set of faces $\E$ is a partition of the mesh skeleton $\bigcup_{K\in\M}\partial K$ into \emph{faces} $\s$ which are subsets contained in hyperplanes of $\overline{\Omega}$.
We denote by $\E_K$ the set of faces of the cell $K$, and we define  $n_{K,\s} \in \R^d$ as the unit normal vector to $\s$ pointing outwards from $K$.	
\end{itemize}
The diameter of a subset $X \subset \overline{\Omega}$ is denoted by $h_X = \sup \lbrace|x-y| \mid (x,y) \in X^2  \rbrace$. 
We define the {mesh size} of $\D$ as $h_\D = \sup \lbrace h_K \mid K \in \M  \rbrace$.
We refer to~\cite[Section 1.1]{DiPDr:20} for more detailed  statements about the mesh and its regularity.
\subsection{Polynomials, discrete unknowns and discrete operators}
In the following, $k$ is a fixed non-negative integer.
First, we introduce polynomial spaces on a subset $X \subset \overline{\Omega}$: 
$\poly^k(X)$ and $\poly^k(X)^d$ denote respectively the spaces of polynomial functions $X \to \R$  and polynomial vector fields $X  \to \R^d$ of degree at most $k$.
Given $Y \subset \overline{X}$, we also define the $L^2$-projector $\Pi_Y^k : C^0(\overline{X}) \to \poly^{k}(Y)$ by the relation
$
	\forall  w \in \poly^k(Y), \, \int_Y \Pi_Y^k(v) w = \int_Y v w.
$
%

We now introduce the set of discrete unknowns corresponding to the mixed-order HHO method \cite{CEP:21,DiPDr:20}, with face unknowns of degree $k$ and (enriched) cells unknowns of degree $k+1$: 
\begin{equation*}
	\V _\D^{k,k+1}  = \left \lbrace \v_{\D} =\big( (v_K )_{K \in \M } ,
			(v_\s)_{\s \in \E} \big) 
			\left |
			\begin{array}{ll}
				\forall K\in\M, & v_K \in \poly^{k+1}(K)  \\
				\forall\s\in\E, & v_\s \in \poly^k(\s)   
			\end{array}				 				
			\right. 			
			\right \rbrace.
\end{equation*}
Given a cell $K \in \M$, we let 
$
	\V_K^{k,k+1} = \poly^{k+1}(K) \times \prod_{\s \in \E_K} \poly^k(\s)
$
be the restriction of $\V_\D^{k,k+1}$ to $K$, and for any generic discrete unknown $\v_{\D} \in \V_{\D}^{k,k+1}$ we denote by $\v_K=\big(v_K,(v_{\s})_{\s\in\E_K}\big)\in\V_K^{k,k+1}$  its local restriction to the cell $K$.
Given any $\v_{\D} \in \V_{\D}^{k,k+1}$, we associate two piecewise polynomial functions $v_\M : \Omega \to \R$ and $v_\E : \bigcup_{K \in \M} \partial K  \to \R$ such that 
\[
	{v_\M}_{|K} = v_K \text{ for all } K \in \M \text{ and } {v_\E}_{|\s} = v_\s \text{ for all } \s \in \E. 
\] 
We also introduce $\one_\D \in \V_\D ^{k,k+1}$ the discrete element such that $1_K = 1$ for any cell $K \in \M$ and $1_\s = 1$ for any face $\s \in \E$.

Now, given a cell $K \in \M$, we define a local \emph{discrete gradient operator} $G_K^k : \V_K^{k,k+1} \to \poly^k(K)^d$ such that, for any $\v_K \in \V_K^{k,k+1}$, $G_K^k(\v_K)$ satisfies
\begin{equation} \label{def:grad}
	\int_K G_K^k(\v_K)  \cdot \tau = 
	\int_K \nabla v_K \cdot \tau 
	+ \sum_{\s \in \E_K} \int_\s (v_\s - v_K) \tau \cdot n_{K,\s} \quad \forall \tau \in  \poly^k(K)^d.
\end{equation}
For any face $\s \in \E_K$, we also define the \emph{jump operator} $J_{K,\s} : \V_K^{k,k+1} \to \poly^k(\s)$ by 
\begin{equation} \label{def:jumps}
	J_{K,\s} (\v_K) = \Pi_\s^k(v_K) - v_\s .
\end{equation}
\subsection{Scheme}
Following the ideas from \cite{CaGui:17,CHHLM:22} our scheme relies on a nonlinear reformulation of Problem \eqref{pb:evol}. To do so, we introduce the logarithm potential $\ell = \log(u)$ and the quasi-Fermi potential $w = \ell + \phi$. At least formally, one has the following relation: 
\begin{equation} \label{eq:nlflux}
  \nabla u + u \nabla \phi = u \nabla \left ( \log(u) + \phi \right )= \e^\ell  \nabla w.
\end{equation}
The scheme relies on this formulation. We will \emph{discretise the potentials as polynomials}, i.e. approximate $\ell$ and $w$ as discrete unknowns in $\V_\D^{k,k+1}$. Then, mimicking the relation $u = \e^\ell,$ we will reconstruct the density thus ensuring its positivity.
Therefore, a solution $\left (\l_\D^n \right)_{n \geq 1}$ to the scheme \eqref{sch:AD} corresponds to an approximation of the logarithms of the solution $u$ (density).

More specifically, for a given discretisation $\l_\D \in \V_\D^{k,k+1}$  of the potential $\ell$, one associates a discrete density ${\u_\D} = (u_\M, u_\E)$ defined as a {pair} of piecewise smooth functions where $u_\M : \Omega \to \R$ corresponds to the cells unknowns and $u_\E : \bigcup_{K \in \M} \partial K  \to \R$ corresponds to the face unknowns, defined as  
\begin{equation} \label{def:ud}
	{u_\M} =  \exp(\ell_\M) 
	\text{ and }
	{u_\E} =  \exp(\ell_\E).
\end{equation}
Note that a discrete density $\u_\D$ is not a collection of polynomials (which is highlighted by the use of the wave {under u}), but it enjoys positivity, both on cells and faces, since it is defined as the exponential of real functions.

Our scheme is based on local contributions on cells, split into { a consistent term and a stabilisation term}. 
Given $K \in \M$ and $\eta_l > 0$, the classical discrete counterpart of $(w,v) \mapsto \int_K \Lambda \nabla w \cdot \nabla v$ is the bilinear form (see \cite[Section 3.2.1]{CEP:21}) 
\[
	a_K : (\w_K, \v_K) \mapsto
	\int_K \Lambda G_K^k(\w_K) \cdot  G_K^k(\v_K)
		+ \eta_l \sum_{\s \in \E_K}  \frac{\Lambda_{K\s}}{h_\s} 
			\int_\s J_{K,\s}(\w_K) J_{K,\s}(\v_K), 
\]	
where {$\Lambda_{K\s} = \|\Lambda_{\mid K} n_{K\s} \cdot n_{K\s}  \|_{L^\infty(\s)}$}.
Similarly, given $\eta_{nl} > 0$, we define a local discretisation of $(\ell,w,v)\mapsto \int_K \e^\ell \Lambda \nabla w \cdot \nabla v$ as a sum of nonlinear consistent \eqref{sch:cons} and stabilisation \eqref{sch:stab_nl} contributions:
\vspace{-0.2cm}
\begin{subequations}\label{sch:loc}
        \begin{empheq}{align}
        	\mathcal{C}_K(\l_K,\w_K, \v_K) 
        		&= \int_K \e^{\ell_K} \Lambda G_K^k(\w_K) \cdot G_K^k(\v_K), \label{sch:cons} \\
 		\mathcal{S}_K(\l_K, \w_K, \v_K) 
 			&=  \eta_{nl} \sum_{\s \in \E_K } \frac{\Lambda_{K\s}}{h_\s} 
				\int_\s  \frac{ \e^{\Pi_\s^k(\ell_K)} + \e^{ \ell_\s}  }{2}
	 			J_{K,\s}(\w_K) J_{K,\s}(\v_K).  \label{sch:stab_nl} 
        \end{empheq}
\end{subequations}
\vspace{-0.2cm}
We can now define a local application $\mathcal{T}_K :  \V_K^{k,k+1} \times \V_K^{k,k+1} \times \V_K^{k,k+1} \to \R$  by 
\begin{equation}
	\mathcal{T}_K (\l_K,\w_K, \v_K) 
		= \mathcal{C}_K(\l_K,\w_K, \v_K) + \mathcal{S}_K(\l_K, \w_K, \v_K)
		+  \varepsilon h_K^{k+2} a_K(\w_K, \v_K),
\label{sch:stationnary}
\end{equation}
where $\varepsilon$ is a non-negative parameter.
At the global level, we define $\mathcal{T}_\D :  \V_\D^{k,k+1} \times \V_\D^{k,k+1} \times \V_\D^{k,k+1} \to \R$  by summing the local contributions: 
\begin{equation}
	\mathcal{T}_\D (\l_\D,\w_\D, \v_\D) 
		= \sum_{K \in \M} \mathcal{T}_K (\l_K,\w_K, \v_K).
\label{sch:stationnary:global}
\end{equation}
%
%
We let $\phid_\D \in \V_\D^{k,k+1}$ be the interpolate of $\phi$: for any $K \in \M$, $\phi_K = \Pi_K^{k+1}(\phi)$ and for all $\s \in \E$, $\phi_\s = \Pi_\s^{k}(\phi)$. 
Now, using a backward Euler discretisation in time with time step $\Delta t > 0$, we introduce the following scheme for \eqref{pb:evol}: 
\newline
find $\left (\l_\D^n \right)_{n \geq 1} \in \left (\V_\D^{k,k+1} \right ) ^\mathbb{N^*}$ such that 
\begin{subequations}\label{sch:AD}
        \begin{empheq}[left = \empheqlbrace]{align}
        	\int_\Omega \frac{u^{n+1}_\M - u^n_\M}{\Delta t} v_\M 
        		&= -  \mathcal{T}_\D(\l_\D^{n+1}, \l_\D^{n+1} +  \phid_\D, \v_\D)
        		&&\forall \v_\D \in \V_\D^{k,k+1}, \label{sch:test}\\
        		u^0_K &= {u^{in}}_{|K}    	&&\forall K\in\M. \label{sch:ini} 
        \end{empheq}
\end{subequations}
Given a solution $\left (\l_\D^n \right)_{n \geq 1}$ to the scheme \eqref{sch:AD},  
as discussed above, we associate a sequence of positive discrete densities $\left (\u_\D^n \right ) _{n \geq 1}$.
\begin{remark}[Parameter $\varepsilon$]
Note that $\mathcal{T}_\D$ is to be understood as a discretisation of $(\ell,w,v) \mapsto \int_\Omega (\e^\ell + \epsilon) \Lambda\nabla  w \cdot \nabla v $, with {$\epsilon \sim \varepsilon h_\D^{k+2}$} a small parameter.
The $\epsilon$ perturbation is used in order to show the existence result of Proposition \ref{prop:exist} and can be seen as a kind of stabilisation. The scaling factor $h_K^{k+2}$ in \eqref{sch:stationnary} is used to get the expected order of convergence.
In practice, numerical results for $\varepsilon = 1$ and $\varepsilon = 0$ are almost the same. 
The influence of this term will be investigated in future works.
\end{remark}

We define the discrete thermal equilibrium as $\u_\D^\infty = ( \rho \e^{-\phi_\M} ,\rho \e^{-\phi_\E}) $, with $\rho = M / \int_\Omega \e^{-\phi_\M}$. One can show that $\u_\D^\infty$ (and the associated logarithm potential $\l_\D ^\infty \in \V_\D^{k,k+1}$) is the only stationary solution to \eqref{sch:AD} with mass $M$.

\section{Main features of the scheme}
In this section, we present some results regarding the analysis of the scheme~\eqref{sch:AD}. 
Given $\l_\D \in \V_\D^{k,k+1}$ a discrete logarithm, we associate a discrete quasi-Fermi potential defined as $\w_\D = \l_\D + \phid_\D - \log(\rho) \one_\D$. 
By definition of $\rho$, one has $w_\M = \log \left ( \frac{u_\M}{u_\M^\infty} \right)$.
Note that, for any $(\l_\D,\v_\D) \in \V_\D^{k,k+1} \times\V_\D^{k,k+1}  $, we have $\mathcal{T}_\D(\l_\D, \l_\D +  \phid_\D, \v_\D) = \mathcal{T}_\D(\l_\D, \w_\D, \v_\D)$.
We now state our fundamental a priori results. 
\begin{prop}[Fundamental a priori relations]
Let $\left (\l_\D^n \right)_{n \geq 1}$ be a solution to the scheme \eqref{sch:AD}, and $\left (\u_\D^n \right ) _{n \geq 1}$ be the associated reconstructed discrete density. 
Then, the following a priori results hold:
\begin{enumerate}
	\item[(i)] the mass is preserved along time: 
$
	\displaystyle \forall n \in \mathbb{N}^*, \, \int_\Omega u^n_\M = \int_\Omega u^{in} = M
$, 
	\item[(ii)] a discrete entropy/dissipation relation holds:
$ \displaystyle	\forall n \in \mathbb{N}, \, \frac{\Entro^{n+1} - \Entro^n }{\Delta t} \leq - \Diss^{n+1}$, 

where the discrete entropy and dissipation are defined by 
$\Entro^n = \int_\Omega u^\infty_\M \Phi_1 \left ( \frac{u^n_\M}{u^\infty_\M}\right )$
and 
$\Diss^{n} = \mathcal{T}_\D (\l^n_\D,\w^n_\D, \w^n_\D) \geq 0$
with $\Phi_1:s \mapsto s \log(s) - s +1$ (and $\Phi_1(0) = 1$).
\end{enumerate}
\end{prop}
\begin{proof}
Using $\one_\D$ as a test function in \eqref{sch:test}, alongside with \eqref{sch:ini}, we get the mass conservation identity (i).
To get (ii), we test \eqref{sch:test} with $\w_\D^{n+1}$,
 and we use the convexity of $\Phi_1$ alongside with the expression of $w_\M^{n+1}$.
\end{proof}
Note that the previous results hold for any $\varepsilon \geq 0$.
Following the ideas of \cite{CaGui:17,CHHLM:22}, the entropy/dissipation relation should allow one to analyse the long-time behaviour of the discrete solutions and to get convergence results. These aspects will be the topics of future works.
We now state an existence result, which holds only for positive $\varepsilon$. The proof follows the strategy used in \cite{CHHLM:22}.
\begin{prop}[Existence of solutions] \label{prop:exist}
Assume that the stabilisation parameter $\varepsilon$ in \eqref{sch:stationnary} is positive.
Then, there exists at least one solution $\left (\l_\D^n \right)_{n \geq 1}$ to the scheme \eqref{sch:AD}. The associated densities $\left (\u_\D^n \right ) _{n \geq 1}$ are positive functions.
\end{prop}
\section{Numerical results} 
The numerical scheme \eqref{sch:AD} requires to solve a nonlinear system of equations at each time step. To do so, we use a Newton method, with an adaptative time stepping strategy: if the Newton method does not converge, we try to compute the solution for a smaller time step $0.5 \times \Delta t$. 
If the method converges, we use for the subsequent time step the value $2 \times \Delta t$. 
The maximal time step allowed is the initial time step.
Each time a linear system has to be solved we perform a static condensation (see \cite[Appendix B.3.2]{DiPDr:20}) in order to eliminate (locally) the cell unknowns. 
Note that the local computations are not implemented in parallel, but only sequentially.
In the sequel, we use the following stabilisation parameters: $\varepsilon = \eta_{nl} = \eta_l = 1$.
\newline
The tests considered below {(on $\Omega = ]0,1[^2$)} are the same as in \cite{CHHLM:22},to which we refer for more detailed explanations and descriptions. 
Given a (face) degree $k$, the scheme \eqref{sch:AD} will be denoted by nlhho\_k, whereas the HFV scheme of \cite{CHHLM:22} will be denoted by nlhfv. 
{Note that nlhho\_0 hinges on affine cell unknowns, whereas the cell unknowns of nlhfv are constant: these two schemes hence do not coincide, and nlhho\_0 is expected to be more costly.}
\subsection{Proof of concept: convergence order and efficiency} \vspace{-0.15cm}
Here, we are interested in the convergence of the scheme when $(h_\D, \Delta t) \to (0,0)$.
To do so, we set the advective potential and diffusion tensor as
$ \phi(x,y) = - x $ and 
$ \Lambda = \begin{psmallmatrix}
l_x & 0 \\ 
0 & 1
\end{psmallmatrix} $ for $l_x>0$.
The exact solution is therefore given by 
\vspace{-0.2cm}
\[
	u(t,x,y) = C_1\e^{-\alpha t + \frac{x}{2}} \left ( 2\pi \cos( \pi x) + \sin(\pi x) \right )
	+ 2 C_1 \pi \e^{ x - \frac{1}{2}  }, 
\]
where $C_1> 0$ and $ \alpha = l_x \left ( \frac{1}{4} + \pi^2  \right )$. Note that $u^{in}$ vanishes on 
$\{ x = 1 \}$, but for any $t > 0$, $u(t, \cdot) > 0$.
Here, our experiments are performed using $l_x =  1$ and $C_1 = 10^{-1}$.
%
%
%
\begin{figure}[ht]
\begin{minipage}[c]{.51\linewidth}
\begin{tikzpicture}[scale= 0.67]
        \begin{loglogaxis}[
            legend style = { 
              at={(0.5,1.1)},
              anchor = south,
              tick label style={font=\footnotesize},
              legend columns=4
            },ylabel=\small{Relative $L^2_t(L^2_x)$-error on the solution},xlabel=\small{Spatial {mesh size} $h_\D$}
          ]
          \addplot table[x=Meshsize,y=error_sol] {graph_data/time_CV_adpt/errors_hho_0};          
          \addplot table[x=Meshsize,y=error_sol] {graph_data/time_CV_adpt/errors_hho_1}; 
          \addplot table[x=Meshsize,y=error_sol] {graph_data/time_CV_adpt/errors_hho_2};             
          \addplot table[x=Meshsize,y=error_sol] {graph_data/time_CV_adpt/errors_hho_3};    
          \logLogSlopeTriangle{0.90}{0.2}{0.1}{2}{black};
          \logLogSlopeTriangle{0.90}{0.2}{0.1}{3}{black};
          \logLogSlopeTriangle{0.90}{0.2}{0.1}{4}{black};
          \logLogSlopeTriangle{0.90}{0.2}{0.1}{5}{black};
          \legend{\tiny nlhho\_0 , \tiny nlhho\_1 , \tiny nlhho\_2 , \tiny nlhho\_3 } 
        \end{loglogaxis}
      \end{tikzpicture}    
\end{minipage}
\begin{minipage}[c]{.51\linewidth}
\begin{tikzpicture}[scale= 0.67]
        \begin{loglogaxis}[
            legend style = { 
              at={(0.5,1.1)},
              anchor = south,
              tick label style={font=\footnotesize},
              legend columns=4
            },ylabel=\small{Relative $L^2_t(L^2_x)$-error on the gradient},xlabel=\small{Spatial {mesh size} $h_\D$}
          ]          
          \addplot table[x=Meshsize,y=error_grad] {graph_data/time_CV_adpt/errors_hho_0};          
          \addplot table[x=Meshsize,y=error_grad] {graph_data/time_CV_adpt/errors_hho_1}; 
          \addplot table[x=Meshsize,y=error_grad] {graph_data/time_CV_adpt/errors_hho_2};            
          \addplot table[x=Meshsize,y=error_grad] {graph_data/time_CV_adpt/errors_hho_3};   
          \logLogSlopeTriangle{0.90}{0.2}{0.1}{1}{black};    
          \logLogSlopeTriangle{0.90}{0.2}{0.1}{2}{black};          
          \logLogSlopeTriangle{0.90}{0.2}{0.1}{3}{black};             
          \logLogSlopeTriangle{0.90}{0.2}{0.1}{4}{black};  
          \legend{\tiny nlhho\_0 , \tiny nlhho\_1 , \tiny nlhho\_2 , \tiny nlhho\_3 }      
        \end{loglogaxis}
      \end{tikzpicture}
\end{minipage}
\vspace{-0.1cm}
\caption{\textbf{Accuracy of transient solutions.} Relative error on triangular meshes.}
\label{fig:CV:evol:order}
\end{figure}
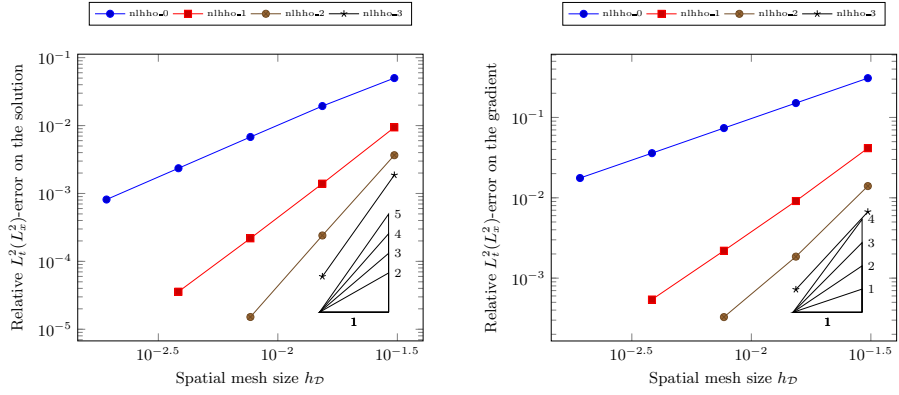
We compute the solution on the time interval $[0,0.1]$, and we denote by $(\u^n_\D)_{1 \leq n \leq N_f}$ the corresponding discrete density. Then, we compute the relative $L^2_t(L^2_x)$ error on the solution and on the gradient of the solution, defined as 
\[
	 \frac{\sqrt{\sum_{n = 1}^{N_f} \delta t^n  \|u_\M^n - u(t^n, \cdot ) \|^2_{L^2(\Omega)} }}
	 	{ \|u\|_{L^2_t(L^2_x)}}
	\text{ and }	
	\frac{\sqrt{ \sum_{n = 1}^{N_f}  \delta t^n  \| \G_\M(\u^n_\D) - \nabla u(t^n, \cdot ) \|^2_{L^2(\Omega)}}}
	{\|\nabla u\|_{L^2_t(L^2_x)}} 
\]
where $\delta t^n = t^{n}-t^{n-1}$ {and the discrete gradient $\G_\M(\u^n_\D)$ is defined by mimicking the continuous relation $\nabla u = \e^\ell \nabla \ell $ as a piecewise continuous function satisfying 		
${\G_\M (\u_\D) }_{| K } = \exp(\ell_K) \,  G_K^k(\l_K)$ on $K \in \M$.}
The $L^2$ norms are computed using quadrature formulas of order $2k +5$.
Note that, with the chosen definitions, we do not take into account the time $t=0$. 
%
%
\begin{figure}[h]
\begin{minipage}[c]{.51\linewidth}
\begin{tikzpicture}[scale= 0.67]
        \begin{loglogaxis}[
            legend style = { 
              at={(0.5,1.1)},
              anchor = south,
              tick label style={font=\footnotesize},
              legend columns=4
            },ylabel=\small{Relative $L^2_t(L^2_x)$-error on the solution},xlabel=\small{Total time of the computation (in $s$)}
          ]
          \addplot table[x=time_cost,y=error_sol] {graph_data/time_CV_adpt/errors_hho_0};          
          \addplot table[x=time_cost,y=error_sol] {graph_data/time_CV_adpt/errors_hho_1};
          \addplot table[x=time_cost,y=error_sol] {graph_data/time_CV_adpt/errors_hho_2};             
          \legend{\tiny nlhho\_0 , \tiny nlhho\_1 , \tiny nlhho\_2  }      
        \end{loglogaxis}
      \end{tikzpicture}    
\end{minipage}
\begin{minipage}[c]{.51\linewidth}
\begin{tikzpicture}[scale= 0.67]
        \begin{loglogaxis}[
            legend style = { 
              at={(0.5,1.1)},
              anchor = south,
              tick label style={font=\footnotesize},
              legend columns=4
            },ylabel=\small{Relative $L^2_t(L^2_x)$-error on the gradient},xlabel=\small{Total time of the computation (in $s$)}
          ]          
          \addplot table[x=time_cost,y=error_grad] {graph_data/time_CV_adpt/errors_hho_0};          
          \addplot table[x=time_cost,y=error_grad] {graph_data/time_CV_adpt/errors_hho_1};          
          \addplot table[x=time_cost,y=error_grad] {graph_data/time_CV_adpt/errors_hho_2};            
          \legend{\tiny nlhho\_0 , \tiny nlhho\_1 , \tiny nlhho\_2 } 
        \end{loglogaxis}
      \end{tikzpicture}
\end{minipage}
\vspace{-0.2cm}
\caption{\textbf{Accuracy vs. computational cost.} Relative errors on triangular meshes.}
\label{fig:CV:evol:cost}
\end{figure}
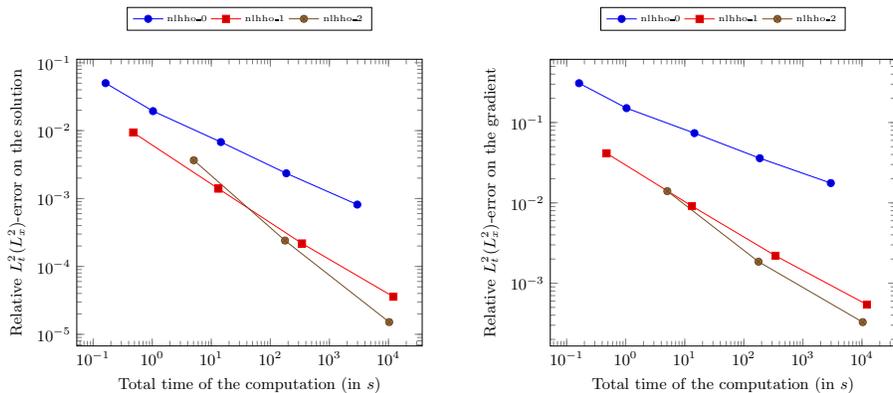
%
To plot the error graphs, we do simulations on a triangular mesh family $(\D_i)_{1\leq i \leq 5}$, such that $h_{\D_{i}} /h_{\D_{i+1}} = 2$.
Since the time discretisation is of order one, on the i-th mesh of the family, we use a time step of $\Delta t_i = \Delta t_k / 2^{(i-1)}$, where $\Delta t_k = 0.05 / 2^{k+2}$ is the initial time step used on $\D_1$. 

In Figure~\ref{fig:CV:evol:order}, we see that the scheme, for face unknowns of degree $k$, converges at order $k+1$ in energy norm and $k+2$ in $L^2$ norm of the density.
In Figure~\ref{fig:CV:evol:cost}, we plot the errors as functions of the computing time to get the solution. 
It is remarkable to see that, even with a low order discretisation in time, significant efficiency gains can be reached by using a high value of $k$. The gain should be even bigger by parallelising the local computations.
{Of course, the use of higher order time-stepping methods should also lead to significant gains, and this should be investigated in future works. However, the way of getting the entropy dissipation relation is currently unclear for such time discretisations.}
\vspace{-0.25cm}
\subsection{Discrete long-time behaviour}
\vspace{-0.15cm}
We are now interested in the long-time behaviour of discrete solutions.
We use the same test-case as before, but with an anisotropic tensor: we set $l_x =  10^{-2}$.
We compute the solution on the time interval $[0,350]$, with $\Delta t = 10^{-1}$, on two Kershaw meshes of sizes $0.02$ and $0.006$. 
\vspace{-0.1cm}
\begin{figure}[h]
\begin{minipage}[c]{.51\linewidth}
\begin{tikzpicture}[scale= 0.65]
        \begin{semilogyaxis}[
            legend style = { 
              at={(0.5,1.1)},
              anchor = south,
              tick label style={font=\footnotesize},
              legend columns=5
            },
            ylabel=\small{$L^1$ distance to $u^\infty$ on the coarsest mesh},
            xlabel=\small{Time}
            ]
          \addplot table[x=Temps,y=Diff_L1] {graph_data/tps_long/mesh4_1_1/tps_nonlin};
          \addplot table[x=Time,y=Diff_L1] {graph_data/tps_long/mesh4_1_1/time_nlhho_0};                   
          \addplot table[x=Time,y=Diff_L1] {graph_data/tps_long/mesh4_1_1/time_nlhho_1};               
          \addplot table[x=Time,y=Diff_L1] {graph_data/tps_long/mesh4_1_1/time_nlhho_2};       				
          \addplot[green] coordinates {
			(0,1)
			(350,4.1484969e-16)
			};
          \legend{\tiny nlhfv , \tiny nlhho\_0,\tiny nlhho\_1, \tiny nlhho\_2, 
           \tiny $ \e^{- \alpha t} $ }          
	      \end{semilogyaxis}
      \end{tikzpicture}    
\end{minipage}
\begin{minipage}[c]{.51\linewidth}
\begin{tikzpicture}[scale= 0.65]
        \begin{semilogyaxis}[
            legend style = { 
              at={(0.5,1.1)},
              anchor = south,
              tick label style={font=\footnotesize},
              legend columns=5
            },ylabel=\small{$L^1$ distance to $u^\infty$ on the finest mesh},xlabel=\small{Time}
          ]
          \addplot table[x=Temps,y=Diff_L1] {graph_data/tps_long/mesh4_1_4/tps_nonlin};          
          \addplot table[x=Time,y=Diff_L1] {graph_data/tps_long/mesh4_1_4/time_nlhho_0};                   
          \addplot table[x=Time,y=Diff_L1] {graph_data/tps_long/mesh4_1_4/time_nlhho_1};         
          \addplot table[x=Time,y=Diff_L1] {graph_data/tps_long/mesh4_1_4/time_nlhho_2};
          \addplot[green] coordinates {
			(0,1)
			(350,4.1484969e-16)
			};
          \legend{\tiny nlhfv , \tiny nlhho\_0,\tiny nlhho\_1, \tiny nlhho\_2, 
           \tiny $ \e^{- \alpha t} $ }       
        \end{semilogyaxis}
      \end{tikzpicture}
\end{minipage}
\caption{\textbf{Long-time behaviour of discrete solutions.} Comparison of the long-time behaviour on Kershaw meshes for $T_f = 350$ and $\Delta t = 0.1$.}
\label{fig:longtime}
\end{figure}
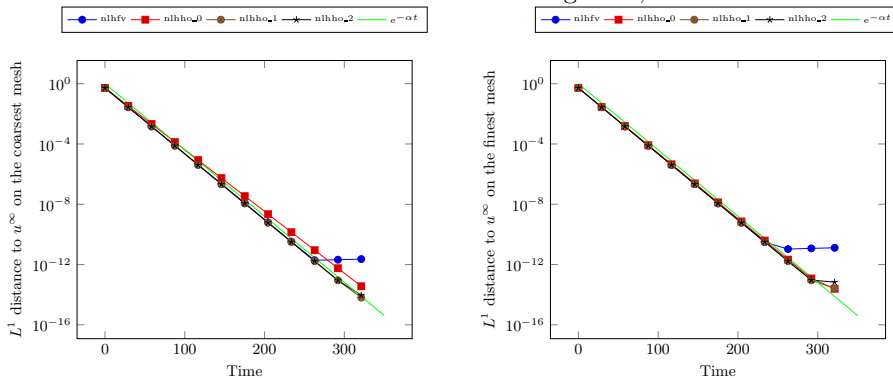
%
In Figure \ref{fig:longtime}, we show the evolution along time of the $L^1$ distance between $\u^n_\D$ and $u^\infty = 2 C_1 \pi \e^{ x - \frac{1}{2} }$ computed as $\int_\Omega |u^n_\M - u^\infty|$.
We observe the exponential convergence towards the steady-state, until some precision is reached. 
The rates of convergence are similar to the exact one ($\alpha$), and do not depend on the size of the mesh.
%

\vspace*{-0.2cm}
\subsection{Positivity}
This last section is dedicated to assessing the discrete positivity preservation.
We set the advection field as
$
\phi(x,y) = -\left (  (x-0.4)^2 +(y-0.6)^2   \right )
$ 
and the diffusion tensor as
$\Lambda = 
	\begin{psmallmatrix}
		0.8 & 0 \\ 
		0 & 1
	\end{psmallmatrix}
$.
For the initial data, we take
$u^{in} = 10^{-3}  \,\1_{B} +   \1_{\Omega \setminus B}$,
where $B$ is the Euclidean ball
$\left \lbrace (x,y) \in \R^2 \mid (x-0.5)^2 + (y-0.5)^2 \leq 0.2^2 \right \rbrace $.
%
%
We perform simulations on the time interval $[0, 5 . 10^{-4}]$ with $\Delta t = 10^{-5}$ on a refined tilted hexagonal-dominant mesh (4192 cells).
%
%
\begin{table}[h]
{\small
\center
\begin{tabularx}{1.\textwidth}{|Y|Y|c|c|c|c|c|} 
\hline 
  	 				& computing time & \#resol	& mincells  		& minfaces 	& mincellQN 	&  minfaceQN \\ 
\hline 
	nlhfv			& 1.77e+01		& 175		& 9.93e-04		& 7.36e-04	& 9.93e-04	&  7.36e-04	\\ 
\hline 
	HMM 				& 2.20e-01		& 50			& -5e-03 		& -7.74e-02	& -5e-03 	& -7.74e-02	\\ 
\hline 
	nlhho\_0 		& 7.17e+01		& 224 		&  1.00e-03		& 1.01-03 	& 2.41e-06 	&  1.01e-03	\\ 
\hline 	
	nlhho\_1 		& 4.13e+02		& 248		&  6.65e-04		& 2.05e-05 	& 1.78e-04 	&  3.57e-08 	\\ 
\hline 	
	nlhho\_2 		& 1.45e+03		& 251		&  9.50e-04		& 5.99e-04 	& 2.67e-07 	&  1.06e-05	\\ 
\hline 		
	nlhho\_3 		& 3.87e+03		& 254		&  9.85e-04		& 8.58e-04 	& 1.10e-05 	&  1.79e-05	\\ 
\hline 	
\end{tabularx} 
}
\caption{\textbf{Positivity of discrete solutions.} 
}\label{table:positivity}
\end{table}
\vspace{-0.4cm}
In Table \ref{table:positivity}, we show the minimal values reached by the schemes. 
The values of ``mincells" are defined as $\min \lbrace \frac{1}{|K|} \int_K u_\M^n  \mid K \in \M, 1 \leq n \leq N_f \rbrace $, whereas 
``mincellQN" are the minimal values taken by the densities on the cell quadrature nodes.
Analogous definitions hold for the faces.
The values of ``\#resol'' correspond to the number of linear systems solved during the computation.
Note that the size of these systems depends on the value of $k$.
The HMM scheme is a linear one (see \cite{CHHLM:22}), therefore only one LU factorisation was performed to compute the solution, which has 90 (resp. 503) negative cell (resp. face) unknowns.
\vspace{0.19cm}
\noindent
\textbf{Acknowledgements} \hspace{0.2cm}
{The author thanks the anonymous reviewers for their remarks and suggestions, as well as Claire Chainais-Hillairet, Maxime Herda and Simon Lemaire for fruitful discussions about this work.
The author acknowledges support by the Labex CEMPI (ANR-11-LABX-0007-01).}
\vspace{-0.4cm}
%
%
%

\end{document}